\documentclass[11pt]{article}

\usepackage{mathrsfs}
\usepackage{amsfonts}
\usepackage{titlesec}
\usepackage{enumerate}
\usepackage[toc,page,title,titletoc,header]{appendix}
\usepackage{amsmath,amsthm,amssymb,anysize}
\usepackage[numbers,sort&compress]{natbib}
\usepackage{color}
\usepackage{graphicx}
\usepackage{threeparttable}

\theoremstyle{definition}
\newtheorem{thm}{Theorem}[section]
\newtheorem{lem}{Lemma}[section]
\newtheorem{rem}[lem]{Remark}

\newtheorem{defn}[lem]{Definition}
\newtheorem{cor}[lem]{Corollary}

\marginsize{3cm}{3cm}{0.5cm}{2.2cm}
\numberwithin{equation}{section}
\numberwithin{equation}{subsection}

\newcommand{\C}{\mathbb{C}}
\newcommand{\F}{\mathbb{F}}
\newcommand{\N}{\mathbb{N}}
\newcommand{\Z}{\mathbb{Z}}

\title{Biderivations,  commuting mappings and (2-)local derivations of $\mathbb{N}$-graded Lie algebras of maximal class}

\author{\textsc{Yong Yang}\\
\small{School of Mathematics and Statistics}\\
\small{Northeast Normal University}\\
\small{130024 Changchun, China}\\
\small{E-mail: yangyong195888221@163.com}
\and
\textsc{Liming Tang}\\
\small{School of Mathematical Sciences}\\
\small{Harbin Normal University}\\
\small{150025 Harbin, China}\\
\small{E-mail: limingtang@hrbnu.edu.cn}
\and
\textsc{Liangyun Chen}\footnote{corresponding author}\\
\small{School of Mathematics and Statistics}\\
\small{Northeast Normal University}\\
\small{130024 Changchun, China}\\
\small{E-mail: chenly640@nenu.edu.cn}}

\date{}

\begin{document}
\maketitle

\begin{quotation}
\small\noindent \textbf{Abstract}:
In Fialowski's classification for algebras of maximal class, there are three Lie algebras of maximal class  with 1-dimensional homogeneous components: $\mathfrak{m}_0$, $L_1$ and $\mathfrak{m}_2$.
In this paper, we studied their biderivations by considering the embedded mapping to derivation algebras.
Then  we determined  commuting mappings on these algebras as an application of biderivations. Finally,
 local and 2-local derivations for these three algebras were  characterized as the given gradings.

\vspace{0.2cm} \noindent{\textbf{Keywords}}: $\mathbb{N}$-graded Lie algebras of maximal class; biderivations; commuting mappings;  (2-)local derivations

\vspace{0.2cm} \noindent{\textbf{Mathematics Subject Classification 2020}}: 17B40, 17B65, 17B70

\end{quotation}
\setcounter{section}{-1}
\section{Introduction}
In the past decades, infinite-dimensional Lie algebras play important roles in the study of Lie theory because of their applications in mathematical physics.
Among infinite-dimensional Lie algebras, $\mathbb{N}$-graded ones have attached much attention.
The theory of $\mathbb{N}$-graded Lie algebras are closely related to nilpotent Lie algebras.
For example, it is obvious that any finite-dimensional $\mathbb{N}$-graded Lie algebra is  nilpotent.
Infinite-dimensional ones are called residual nilpotent Lie algebras.
Shalev and Zelmanov \cite{cc} introduced the definition of coclass of a
finitely generated and residually nilpotent Lie algebra $\mathfrak{g}$, in analogy with the
case of (pro-)$p$-groups, as $cc(\mathfrak{g})=\sum_{i\geq1}(\mathrm{dim}(\mathfrak{g}^{i}/\mathfrak{g}^{i+1})-1)$
(possibly infinity), where $\mathfrak{g}^{i}$ is the lower central series of $\mathfrak{g}$.
Lie algebras of coclass 1 are called Lie algebras of maximal
class (also called narrow or thin Lie algebras).
Fialowski \cite{fc} classified all infinite-dimensional $\mathbb{N}$-graded two-generated Lie
algebras $\mathfrak{g}=\bigoplus_{i=1}^{\infty}\mathfrak{g}_i$
with 1-dimensional homogeneous components $\mathfrak{g}_i$. According to her  classification, we obtain  that up to isomorphism, there are only three $\mathbb{N}$-graded Lie algebras of maximal class with 1-dimensional homogeneous components: $\mathfrak{m}_0$, $L_1$ and $\mathfrak{m}_2$, where  $L_1$ is the positive part of the Witt algebra $\mathrm{Der}(\mathbb{F}[x])$. This result  was also rediscovered in \cite{cc}.  Furthermore, the cohomology of these algebras with coefficients in the trivial modules was studied in \cite{ft}. The adjoint cohomology and deformations were studied in \cite{m0,L1,m2,M1,M2}.

The theory of biderivations and commuting mappings was introduced in \cite{b} for associative algebras, developed for Lie algebras in \cite{c2016,ccz2019,dt2019,e2021,jt2020} and Lie superalgebras in \cite{ccc2021,tmc2020,ycc2021}. It happens quite often that all biderivations are inner, for example, see \cite{ex1,ex2,ex4,ccc2021,ccz2019,ycc2021,zcz2020}. In particular, it was proved that all biderivations of a finite-dimensional complex simple Lie algebra, without the restriction of being skew-symmetric, are inner biderivations in \cite{ex1}. However, compared with the finite-dimensional case, less work has done for infinite-dimensional Lie algebras. In this paper,
 biderivations and commuting mappings of $\mathfrak{m}_0$, $L_1$ and $\mathfrak{m}_2$ were studied.
Firstly, we determined  the biderivations by considering the embedded mappings from  these algebras to their derivation algebras. Moveover, we characterized commuting mappings by the biderivatons.

Another generalized derivations that we are also interested in  are local and 2-local derivations.
The concepts of local  and 2-local derivations were introduced in \cite{K}  and \cite{P}.
 In recent years, local and 2-local derivations have aroused the interest of  many authors, see
\cite{AKR,l1,M0}. In particular,  it is proved that a finite-dimensional nilpotent Lie algebra $L$ with $\mathrm{dim}\ L\geq2$ always admits a 2-local derivation which is not a derivation in \cite{AKR}. However,
it does not hold for infinite-dimensional $\mathbb{N}$-graded Lie algebras. In fact, every  2-local derivation of $L_1$ is a  derivation \cite{l1}. In this paper, we proved that every local and 2-local derivation is a derivation for $\mathfrak{m}_0$, $L_1$ and $\mathfrak{m}_2$.

\section{Preliminaries}
Throughout this paper, the ground field $\mathbb{F}$ is an algebraically closed field of characteristic
zero and all vector spaces, algebras are over $\mathbb{F}$. A Lie algebra over $\mathbb{F}$ is a skew-symmetric  algebra whose multiplication satisfies the Jacobi identity.
Lie algebras of coclass 1 are called \emph{Lie algebras of maximal
class}.
Up to isomorphism, there are only three
$\mathbb{N}$-graded Lie algebras of maximal class with 1-dimensional homogeneous
components (see \cite{fc}):

$\bullet$ $\mathfrak{m}_0$: the Lie algebra $\mathfrak{m}_0$ is an $\mathbb{N}$-graded Lie algebra $\mathfrak{m}_{0}=\bigoplus_{i=1}^{\infty} (\mathfrak{m}_0)_i$ with 1-dimensional graded components $(\mathfrak{m}_0)_i$ and generated by the  components of degree 1 and 2. For a basis $e_i$ of $(\mathfrak{m}_0)_i$, the non-trivial brackets are $[e_1,e_i]=e_{i+1}$ for all $i\geq 2$.

$\bullet$ $L_1$: the Lie algebra $L_{1}$ is an $\mathbb{N}$-graded Lie algebra $L_{1}=\bigoplus_{i=1}^{\infty}L_{1}^{(i)}$
with 1-dimensional graded components $L_{1}^{(i)}$  and generated by the  components of degree 1 and 2. For a basis $e_i$ of
$L_{1}^{(i)}$, the non-trivial brackets are $[e_i,e_j]=(j-i)e_{i+j}$ for all $i,j\geq1$.

$\bullet$ $\mathfrak{m}_2$: the Lie algebra $\mathfrak{m}_2$ is an $\mathbb{N}$-graded Lie algebra $\mathfrak{m}_2=\bigoplus_{i=1}^{\infty}(\mathfrak{m}_2)_i$ with 1-dimensional graded components
$(\mathfrak{m}_2)_i$ and generated by the  components of degree 1 and 2. For a basis $e_i$ of
$(\mathfrak{m}_2)_i$, the non-trivial brackets are $[e_1,e_i]=e_{i+1}$ for all $i\geq2$, $[e_2,e_j]=e_{j+2}$ for all $j\geq3$.

\begin{defn}
A linear mapping of a Lie algebra $\mathfrak{g}$ is called a \emph{derivation} if it satisfies
$$D([x,y])=[D(x),y]+[x,D(y)],$$
for all $x,y$ in $\mathfrak{g}$. Denote by $\mathrm{Der}(\mathfrak{g})$ the set of derivations of $\mathfrak{g}$.
\end{defn}
For  $x\in \mathfrak{g}$, the mapping $\mathrm{ad}\ x: y\mapsto [x,y]$ is a derivation of $\mathfrak{g}$. Call it \emph{inner derivation}. A standard fact is that $\mathrm{Der}(\mathfrak{g})$ and $\mathrm{ad}\ \mathfrak{g}$ are both Lie subalgebra of $\mathrm{Hom}(\mathfrak{g},\mathfrak{g})$ and $\mathrm{ad}\ \mathfrak{g}$ is an ideal of $\mathrm{Der}(\mathfrak{g})$.
Moreover, the first cohomology with coefficients in the adjoint module is defined by
$\mathrm{H}^{1}(\mathfrak{g},\mathfrak{g})=\mathrm{Der}(\mathfrak{g})/\mathrm{ad}\ \mathfrak{g}$ (see \cite{Fuks}).
A bilinear mapping $f$ of $\mathfrak{g}$ is called \emph{skew-symmetric}  if
$$f(x,y)=-f(y,x),$$
for all $x,y$ in $\mathfrak{g}$. For a bilinear mapping $f$ of $\mathfrak{g}$ and an element $x$ in $\mathfrak{g}$, we define $L_{f,x}$ and $R_{f,x}$ by two linear mappings of $\mathfrak{g}$ satisfying
$L_{f,x}(y)=f(x,y)$ and $R_{f,x}(y)=f(y,x)$ for $y\in\mathfrak{g}$.
\begin{defn}\label{b}
A skew-symmetric bilinear mapping $f$ of a Lie algebra $\mathfrak{g}$ is called a \emph{biderivation} of $\mathfrak{g}$ if
$L_{f,x}$ and $R_{f,x}$ both are derivations of $\mathfrak{g}$ for any $x$ in $\mathfrak{g}$.
Denote by $\mathrm{BDer}(\mathfrak{g})$ the set of  biderivations of $\mathfrak{g}$.
\end{defn}
Suppose that the mapping $f_{\lambda}: \mathfrak{g}\times \mathfrak{g}\longrightarrow \mathfrak{g}$ is defined by $f_{\lambda}(x,y)=\lambda [x,y]$
for all $x,y\in\mathfrak{g}$, where $\lambda\in\mathbb{F}$. Then it is easy to check that $f_{\lambda}$ is a biderivation of $\mathfrak{g}$. This class of biderivations is called \emph{inner biderivation}.
Denote by $\mathrm{IBDer}(\mathfrak{g})$ the set of inner biderivations of $\mathfrak{g}$.

Suppose that $\mathfrak{g}=\bigoplus_{i=1}^{\infty} \mathfrak{g}_i$ is a $\mathbb{N}$-graded Lie algebra.
For $x\in\mathfrak{g}$, write $\|x\|$ for the degree of $x$.
 If $\mathfrak{g}=\bigoplus_{i=1}^{\infty} \mathfrak{g}_i$ has a homogenous basis $\{e_i\mid i\in\mathbb{N}\}$, we define $e_{i}^{j}\in\mathrm{Hom}(\mathfrak{g},\mathfrak{g})_{j-i}$ and $e^{i,j}_{k}\in\mathrm{Hom}(\mathfrak{g}\wedge\mathfrak{g},\mathfrak{g})_{k-i-j}$, $i<j$ by
$e_{i}^{j}:e_i\rightarrow e_j,$ and $e^{i,j}_{k}:(e_i,e_j)\rightarrow e_k$.
 It is easy to see that $\mathrm{Der}(\mathfrak{g})$ and $\mathrm{BDer}(\mathfrak{g})$  are $\mathbb{N}$-graded subspaces of $\mathrm{Hom}(\mathfrak{g},\mathfrak{g})$ and $\mathrm{Hom}(\mathfrak{g}\wedge\mathfrak{g},\mathfrak{g})$ respectively, where the
homogeneous components of weight $k$ are given by
$$\mathrm{Der}_{k}(\mathfrak{g})=\{\phi\in\mathrm{Der(\mathfrak{g})}\mid\phi(\mathfrak{g}_i)\subseteq \mathfrak{g}_{k+i},i\in\mathbb{N}\}$$
and
$$\mathrm{BDer}_{k}(\mathfrak{g})=\{f\in\mathrm{BDer(\mathfrak{g})}\mid f(\mathfrak{g}_i,\mathfrak{g}_j)\subseteq \mathfrak{g}_{k+i+j},i,j\in\mathbb{N}\}.$$
\begin{defn}
For a Lie algebra $\mathfrak{g}$, we denote the \emph{center} of
$\mathfrak{g}$ by
$$\mathrm{C}(\mathfrak{g})=\{x\in\mathfrak{g}\mid [x,\mathfrak{g}]=0\}.$$
The  Lie algebra $\mathfrak{g}$ is called \emph{centerless} if $\mathrm{C}(\mathfrak{g})=0$.
\end{defn}
\begin{rem}
It is easy to see that $\mathfrak{m}_0$, $L_1$ and $\mathfrak{m}_2$ are all centerless.
\end{rem}

\begin{lem}\label{C}
Suppose that $\mathfrak{g}$ is a centerless $\mathbb{N}$-graded Lie algebra. Then
$\mathfrak{g}$  can be embedded into
 a   $\mathbb{Z}$-graded Lie algebra $\mathfrak{\widetilde{g}}$ as an ideal such that
 $$\mathrm{Der}(\mathfrak{g})\cong \mathrm{ad}_\mathfrak{g}\ \mathfrak{\widetilde{g}}, \quad
 \mathrm{Ann}_{\mathfrak{\widetilde{g}}}\ \mathfrak{g}=
 \{x\in\mathfrak{\widetilde{g}}\mid [x,\mathfrak{g}]=0\}=0.$$
 Moreover, for any $f\in \mathrm{BDer}_{k}(\mathfrak{g})$, $k\in\Z$, there exists a unique linear mapping $\varphi_f:\ \mathfrak{g}\rightarrow \mathfrak{\widetilde{g}}$ such that, for $x,y\in\mathfrak{g}$,
$$f(x,y)=[\varphi_f(x),y]=-[\varphi_f(y),x],$$
where $\|f\|=\|\varphi_f\|$.
\end{lem}
\begin{proof}
Since $\mathfrak{g}$ is centerless, we have $\mathfrak{g}\cong \mathrm{ad}\ \mathfrak{g}\triangleleft\mathrm{Der}(\mathfrak{g})$.
Let  $D\in \mathrm{Der}(\mathfrak{g})$. If $[D,\mathrm{ad}\ x]=\mathrm{ad}\ D(x)=0$ for all $x\in \mathfrak{g}$, then $D(x)\in \mathrm{C}(\mathfrak{g})=0$. So $D=0$.
By identifying
$\mathfrak{g}$ with $\mathrm{ad}\ \mathfrak{g}$, we get the isomorphism $\mathrm{Der}(\mathfrak{g})\cong \mathrm{ad}_\mathfrak{g}(\mathrm{Der}(\mathfrak{g}))$.
Set $\mathfrak{\widetilde{g}}=\mathrm{Der}(\mathfrak{g})$. We get $\mathrm{Der}(\mathfrak{g})\cong \mathrm{ad}_\mathfrak{g}\ \mathfrak{\widetilde{g}}$ and $\mathrm{Ann}_{\mathfrak{\widetilde{g}}}\ \mathfrak{g}=0$.
For $f\in \mathrm{BDer}(\mathfrak{g})$ and $x\in\mathfrak{g}$, the mapping $L_{f,x}\in\mathrm{Der}(\mathfrak{g})$ by Definition \ref{b}. Then there exists $y_x\in
\mathfrak{\widetilde{g}}$ such that $L_{f,x}=\mathrm{ad}_\mathfrak{g}\ y_x$. This $y_x$ is unique since $\mathrm{Ann}_{\mathfrak{\widetilde{g}}}\ \mathfrak{g}=0$. So we get a unique  linear mapping $\varphi_f:\
\mathfrak{g}\rightarrow\mathfrak{\widetilde{g}}$ denoted by $\varphi_f(x)=y_x$. Then,
 for any $x,y\in\mathfrak{g}$, we have
$f(x,y)=-f(y,x)=L_{f,x}(y)=[\varphi_{f}(x),y]=-[\varphi_{f}(y),x].$
\end{proof}

\begin{defn}
A linear mapping $\phi$ of a Lie algebra $\mathfrak{g}$ is called \emph{commuting}  if
$[\phi(x),x]=0$ for any $x$ in $\mathfrak{g}$.
\end{defn}
An important application of commuting mappings is to construct biderivations (for example, see \cite{ex1,ex2,ex4}), as shown in the following lemma.
\begin{lem}\label{linear}\cite{ex4}
Suppose that $\mathfrak{g}$ is a Lie algebra and $\phi$ is a commuting mapping of $\mathfrak{g}$. Then the bilinear form $\phi_f$, satisfying $\phi_{f}(x,y)=[x,\phi(y)]$ for $x,y\in \mathfrak{g}$, is a  biderivation of $\mathfrak{g}$.
\end{lem}
We here  recall and introduce theories of local and 2-local derivations of Lie algebras.

\begin{defn}
A  linear mapping $\Delta$ of a Lie algebra $\mathfrak{g}$ is called a \emph{local derivation} of $\mathfrak{g}$ if
for
every $x\in\mathfrak{g}$, there exists a derivation $D_{x}$ of $\mathfrak{g}$ (depending on $x$) such that
$\Delta(x)=D_{x}(x)$.
Denote by $\mathrm{LDer}(\mathfrak{g})$ the set of 2-local derivations of $\mathfrak{g}$.
\end{defn}

\begin{defn}\label{bl}
A  linear mapping $\Delta$ of a Lie algebra $\mathfrak{g}$ is called a \emph{ 2-local derivation} of $\mathfrak{g}$ if
for
every $x,y\in\mathfrak{g}$, there exists a derivation $D_{x,y}$ of $\mathfrak{g}$ (depending on $x,y$) such that
$\Delta(x)=D_{x,y}(x)$ and $\Delta(y)=D_{x,y}(y)$.
Denote by $\mathrm{BLDer}(\mathfrak{g})$ the set of 2-local derivations of $\mathfrak{g}$.
\end{defn}
\begin{rem}
Obviously, a derivation is a 2-local derivation.
By taking $x=y$ in Definition \ref{bl}, we get that a  2-local derivation is a local derivation automatically.
Therefore,
for a Lie algebra $\mathfrak{g}$, we have
$$\mathrm{LDer}(\mathfrak{g})\supseteq\mathrm{BLDer}(\mathfrak{g})\supseteq\mathrm{Der}(\mathfrak{g}).$$
\end{rem}

\begin{lem}\label{ld}
Suppose that $\mathfrak{g}=\bigoplus_{i=1}^{\infty} \mathfrak{g}_i$ is an $\mathbb{N}$-graded Lie algebra. For $k\in\mathbb{Z}$, Set
\begin{eqnarray*}
  \mathrm{LDer}_k(\mathfrak{g})&=& \{\Delta\in\mathrm{Hom}(\mathfrak{g},\mathfrak{g})|\ \mathrm{for\ any}\ x\in\mathfrak{g}, \mathrm{there\ exists}\ D_{x;k}\in\mathrm{Der}_{k}(\mathfrak{g}), \mathrm{such\ that}\ \\
   && \Delta(x)=D_{x;k}(x)\}.
\end{eqnarray*}
Then $\mathrm{LDer}(\mathfrak{g})=\sum_{k\in\mathbb{Z}}\mathrm{LDer}_k(\mathfrak{g})$.
\end{lem}
\begin{proof}
For a local derivation $\Delta\in \mathfrak{g}$, it is sufficient to prove that $\Delta\in\sum_{k\in\mathbb{Z}}\mathrm{LDer}_k(\mathfrak{g})$. For any $x\in\mathfrak{g}$, by definition, there exists a derivation $D_x\in\mathrm{Der}(\mathfrak{g})$ and $D_{x;k}\in\mathrm{Der}_k(\mathfrak{g})$ such that
$$\Delta(x)=D_x(x)=\sum_{k\in\mathbb{Z}}D_{x;k}(x).$$
Define a mapping $\Delta_k$ by $\Delta_k(x)=D_{x;k}(x)$ for any $x\in\mathfrak{g}$. Then $\Delta_k\in
\mathrm{LDer}_k(\mathfrak{g})$ and
$\Delta(x)=\sum_{k\in\mathbb{Z}}\Delta_k(x)$. Thus, $\Delta=\sum_{k\in\mathbb{Z}}\Delta_k\in\sum_{k\in\mathbb{Z}}\mathrm{LDer}_k(\mathfrak{g})$.
\end{proof}

Next, we characterize all  biderivations, linear commuting mappings and (2-)local derivations of
$\mathfrak{m}_{0}$, $L_1$ and $\mathfrak{m}_{2}$ one by one.

\section{$\mathbb{N}$-graded Lie algebra  $\mathfrak{m}_{0}$}
From the result of the first cohomology of $\mathfrak{m}_0$  with coefficients in the adjoint module \cite[Theorem 2]{m0}, we get the derivations of $\mathfrak{m}_0$ by considering the inner derivations in the following lemma.
\begin{lem}\label{m0}
$\mathrm{dim}\ \mathrm{Der}_{k}(\mathfrak{m}_0)=\left\{
                    \begin{array}{ll}
                      2, & \hbox{$k\geq0$;} \\
                      0, & \hbox{$k\leq-1$.}
                    \end{array}
                  \right.
$
In particular,

(1) in $\mathrm{Der}_{k}(\mathfrak{m}_0)$ for $k\geq1$, a basis is $e^{1+k}_1$, $\sum_{i\geq2}e^{i+k}_i$;

(2) in $\mathrm{Der}_{0}(\mathfrak{m}_0)$, a basis is $\sum_{i\geq2}e^{i}_i$, $e^{1}_{1}+\sum_{i\geq3}(i-2)e^{i}_i$.
\end{lem}

In order to describe the derivation algebra of $\mathfrak{m}_0$, we denote by $$\mathfrak{\widetilde{m}}_0=\mathfrak{m}_0\oplus\mathrm{span}\{x_{01},x_{02},x_i,i\geq1\}$$
 the Lie algebra with brackets:
$$[e_1,e_i]=e_{i+1},\ [x_{01},x_1]=-x_1,\ [x_{01},x_k]=kx_k,\ [x_{02},x_1]=x_1,\ [x_1,x_k]=-e_{k+1},$$
$$[x_{01},e_1]=e_1,\ [x_{01},e_i]=(i-2)e_i,\
[x_{02},e_i]=e_i,\ [x_1,e_1]=-e_2,\ [x_k,e_i]=e_{i+k},$$
where $k,i\geq2$.
Obviously, $\mathfrak{\widetilde{m}}_0=\bigoplus^{\infty}_{i=0}(\mathfrak{\widetilde{m}}_0)_i$ is a $\mathbb{Z}$-graded Lie algebra with the graded component
$(\mathfrak{\widetilde{m}}_0)_i=\left\{
                                  \begin{array}{ll}
                                    \mathrm{span}\{x_{01},x_{02}\}, & \hbox{$i=0$;} \\
                                    \mathrm{span}\{e_{i},x_{i}\}, & \hbox{$i\geq 1$.}
                                  \end{array}
                                \right.$
By Lemma \ref{m0}, $\mathrm{Der}(\mathfrak{m}_0)\cong\mathfrak{\widetilde{m}}_0$.

\begin{thm}\label{M01}
$\mathrm{dim}\ \mathrm{BDer}_{k}(\mathfrak{m}_0)=\left\{
                    \begin{array}{ll}
                      1, & \hbox{$k\geq-1$;} \\
                      0, & \hbox{$k\leq-2$.}
                    \end{array}
                  \right.
$
In particular, for $k\geq-1$, $\mathrm{BDer}_{k}(\mathfrak{m}_0)$ is spanned by $\sum_{i\geq2}e^{1,i}_{1+i+k}$.
\end{thm}

\begin{proof}
By Lemma \ref{C},
for $f\in \mathrm{BDer}_{k}(\mathfrak{m}_0)$, $k\in\Z$, there exists a  linear mapping $\varphi_f:\ \mathfrak{m}_0\rightarrow \mathfrak{\widetilde{m}}_0$ such that, for $i,j\geq1$,
\begin{equation}\label{mb}
f(e_i,e_j)=[\varphi_f(e_i),e_j]=-[\varphi_f(e_j),e_i],
\end{equation}
where $\|f\|=\|\varphi_f\|=k$. Now we determine $\varphi_f$ in different weight $k$.

\begin{flushleft}
$\mathbf{Case\ 1.}\ k\leq-2.$
\end{flushleft}
In this case, $\varphi_f(e_1)=\cdots=\varphi_f(e_{-k-1})=0$. Set $\varphi_f(e_{-k})=\alpha x_{01}+\beta x_{02}$ and $\varphi_f(e_{-k+i})=\lambda_ie_i+\mu_ix_i$ for $i\geq 1$.
Taking $j=1$ in Eq. (\ref{mb}), we get
$[\varphi_f(e_{i}),e_1]=0$.
Thus, we have
\begin{eqnarray*}
&&[\varphi_f(e_{-k}),e_1]=[\alpha x_{01}+\beta x_{02},e_1]=\alpha e_1=0, \\
&&[\varphi_f(e_{-k+1}),e_1]=[\lambda_1e_1+\mu_1x_1,e_1]=-\mu_1e_2=0,\\
&&[\varphi_f(e_{-k+i}),e_1]=[\lambda_ie_i+\mu_ix_i,e_1]=-\lambda_ie_{i+1}=0,
\end{eqnarray*}
where $i\geq 2$. So $\alpha=\mu_1=\lambda_i=0$ for $i\geq 2$. Taking $i=j$ in Eq. (\ref{mb}), we get
$[\varphi_f(e_{i}),e_i]=0$. Thus we have
\begin{eqnarray*}
 &&[\varphi_f(e_{-k}),e_{-k}]=[\beta x_{02},e_{-k}]=\beta e_{-k}=0, \\
 &&[\varphi_f(e_{-k+1}),e_{-k+1}]=[\lambda_1e_1+\mu_1x_1,e_{-k+1}]=\lambda_1e_{-k+2}=0,\\
&&[\varphi_f(e_{-k+i}),e_{-k+i}]=[\lambda_ie_i+\mu_ix_i,e_{-k+i}]=\mu_ie_{-k+2i}=0,
\end{eqnarray*}
where $i\geq2$. So $\beta=\lambda_1=\mu_i=0$ for $i\geq 2$. Thus $f=\varphi_f=0$.

\begin{flushleft}
$\mathbf{Case\ 2.}\ k=-1.$
\end{flushleft}
In this case, we set $\varphi_f(e_1)=\alpha x_{01}+\beta x_{02}$ and $\varphi_f(e_i)=\lambda_{i-1}e_{i-1}+\mu_{i-1}x_{i-1}$ for $i\geq 2$.
Taking $i=j$ in Eq. (\ref{mb}), we get
$[\varphi_f(e_{i}),e_i]=0$. Thus we have
\begin{eqnarray*}
 &&[\varphi_f(e_{1}),e_{1}]=[\alpha x_{01}+\beta x_{02},e_{1}]=\alpha e_{1}=0, \\
 &&[\varphi_f(e_{2}),e_{2}]=[\lambda_{1}e_{1}+\mu_{1}x_{1},e_{2}]=\lambda_1e_3=0,\\
&&[\varphi_f(e_{i}),e_{i}]=[\lambda_{i-1}e_{i-1}+\mu_{i-1}x_{i-1},e_{i}]=\mu_{i-1}e_{2i-1}=0,
\end{eqnarray*}
where $i\geq3$. So $\alpha=\lambda_1=\mu_{i-1}=0$ for $i\geq 3$.
Taking $i=1$ in Eq. (\ref{mb}), we get
$[\varphi_f(e_{1}),e_j]=-[\varphi_f(e_{j}),e_1]$.
Thus, we have
\begin{eqnarray*}
   &&[\varphi_f(e_{1}),e_2]=-[\varphi_f(e_{2}),e_1]=\beta e_2=\mu_1e_2, \\
  &&[\varphi_f(e_{1}),e_j]=-[\varphi_f(e_{j}),e_1]=\beta e_j=\lambda_{j-1}e_j,
\end{eqnarray*}
where $j\geq 3$. So $\beta=\mu_1=\lambda_{j-1}$ for $j\geq 3$.
Thus,
\begin{eqnarray*}
&&f(e_1,e_i)=[\varphi_f(e_1),e_i]=[\beta x_{02},e_i]=\beta e_i,\ i\geq2, \\
&&f(e_2,e_i)=[\varphi_f(e_2),e_i]=[\beta x_1,e_i]=0 ,\ i\geq 3,\\
&&f(e_i,e_j)=[\varphi_f(e_i),e_j]=[\beta e_{i-1},e_j]=0  ,\ j>i\geq 3.
\end{eqnarray*}
That is,
 $f=\beta\sum_{i\geq2}e^{1,i}_{i}$.

\begin{flushleft}
$\mathbf{Case\ 3.}\ k=0.$
\end{flushleft}
In this case, we set  $\varphi_f(e_i)=\lambda_{i}e_{i}+\mu_{i}x_{i}$ for $i\geq 1$.
Taking $i=j$ in Eq. (\ref{mb}), we get
$[\varphi_f(e_{i}),e_i]=0$. Thus we have
\begin{eqnarray*}
   &&[\varphi_f(e_{1}),e_1]=[\lambda_{1}e_{1}+\mu_{1}x_{1},e_1]=-\mu_1 e_2=0, \\
  &&[\varphi_f(e_{i}),e_i]=[\lambda_{i}e_{i}+\mu_{i}x_{i},e_i]=\mu_{i}e_{2i}=0,
\end{eqnarray*}
where $i\geq 2$. So $\mu_i=0$ for $i\geq 1$. Taking $i=1$ in Eq. (\ref{mb}), we get
$[\varphi_f(e_1),e_j]=-[\varphi_f(e_j),e_1]$. Thus, we have
\begin{equation*}
[\varphi_f(e_1),e_j]=-[\varphi_f(e_j),e_1]=\lambda_1e_{j+1}=\lambda_je_{j+1},
\end{equation*}
where $j\geq 2$. So $\lambda_1=\lambda_j$ for $j\geq 1$. Thus,
\begin{eqnarray*}
&&f(e_1,e_i)=[\varphi(e_1),e_i]=[\lambda_1e_1,e_i]=\lambda_1e_{i+1},\ i\geq2,\\
&&f(e_i,e_j)=[\varphi(e_i),e_j]=[\lambda_1e_{i},e_j]=0,\ j>i\geq2.
\end{eqnarray*}
That is,  $f=\lambda_1\sum_{i\geq 2}e^{1,i}_{1+i}$.

\begin{flushleft}
$\mathbf{Case\ 4.}\ k\geq1.$
\end{flushleft}
In this case, we set  $\varphi_f(e_i)=\lambda_{i}e_{i+k}+\mu_{i}x_{i+k}$ for $i\geq 1$.
Taking $i=j$ in Eq. (\ref{mb}), we get
$[\varphi_f(e_{i}),e_i]=0$. Thus we have
\begin{eqnarray*}
   &&[\varphi_f(e_{1}),e_1]=[\lambda_{1}e_{k+1}+\mu_{1}x_{k+1},e_1]=-\lambda_1 e_{k+2}=0, \\
  &&[\varphi_f(e_{i}),e_i]=[\lambda_{i}e_{k+i}+\mu_{i}x_{k+i},e_i]=\mu_{i}e_{k+2i}=0,
\end{eqnarray*}
where $i\geq 2$. So $\lambda_1=\mu_i=0$ for $i\geq 2$. Taking $i=1$ in Eq.
(\ref{mb}), we get
$[\varphi_f(e_1),e_j]=-[\varphi_f(e_j),e_1]$. Thus, we have
\begin{equation*}
[\varphi_f(e_1),e_j]=-[\varphi_f(e_j),e_1]=\mu_1e_{k+j+1}=\lambda_je_{k+j+1},
\end{equation*}
where $j\geq 2$. So $\mu_1=\lambda_j$ for $j\geq 2$. Thus,
\begin{eqnarray*}
&&f(e_1,e_i)=[\varphi(e_1),e_i]=[\mu_1e_1,e_i]=\mu_1e_{i+1},\ i\geq2,\\
&&f(e_i,e_j)=[\varphi(e_i),e_j]=[\mu_1e_{i+k},e_j]=0,\ j>i\geq2.
\end{eqnarray*}
That is, $f=\mu_1\sum_{i\geq 2}e^{1,i}_{1+i+k}$.

In conclusion,
$$f=\left\{
                    \begin{array}{ll}
                      \lambda_k\sum_{i\geq 2}e^{1,i}_{1+i+k}, & \hbox{$\lambda_k\in\mathbb{F},\ k\geq-1$;} \\
                      0, & \hbox{$k\leq-2$.}
                    \end{array}
                  \right.
$$
The proof is complete.
\end{proof}

\begin{cor}
$\mathrm{BDer}_{0}(\mathfrak{m}_0)=\mathrm{IBDer}(\mathfrak{m}_0)$.
\end{cor}

\begin{thm}
A linear mapping of $\mathfrak{m}_0$ is commuting if and only if it is a scalar of $\sum_{i\geq1}e^{i}_{i}$.
\end{thm}
\begin{proof}
The `if' direction is easy to verify. We now prove the `only if' direction.
Suppose that $\phi$ is linear commuting mapping of weight $k$. That is to say that $\phi(e_i)=a_ie_{i+k}$ for $i\geq1$. By Lemma \ref{linear}, $\phi$ defines a biderivation $\phi_{f}\in\mathrm{BDer}_{k}(\mathfrak{m}_0)$.
If $k\leq-2$, then $\phi_{f}=0$ by Theorem \ref{M01}. So $[x,\phi(y)]=0$ for all $x,y\in\mathfrak{m}_0$. From $\phi(y)\in\mathrm{C}(\mathfrak{m}_0)=0$, we have $\phi=0$. If $k\geq-1$, then
$\phi_{f}=\lambda\sum_{i\geq2}e^{1,i}_{1+i+k}$ for some $\lambda\in\mathbb{F}$.
\begin{flushleft}
$\mathbf{Case\ 1.}\ k\geq-1, k\neq 0.$
\end{flushleft}
From $\phi_{f}(e_j,e_1)=\lambda\sum_{i\geq2}e^{1,i}_{1+i+k}(e_j,e_1)$ for $j\geq2$, we have
$[e_j,a_1e_{k+1}]=-\lambda e_{1+j+k}$.
Since $k\neq 0$, we have $[e_j,e_{k+1}]=0$. So $\lambda=0$. We have $\phi_{f}(x,y)=[x,\phi(y)]=0$ for all $x,y\in\mathfrak{m}_0$. So $\phi(y)\in \mathrm{C}(\mathfrak{m}_0)=0$. Then $\phi=0$.
\begin{flushleft}
$\mathbf{Case\ 2.}\ k=0.$
\end{flushleft}
From $\phi_{f}(e_1,e_j)=\lambda\sum_{i\geq2}e^{1,i}_{1+i}(e_1,e_j)$ for $j\geq2$,
we have $[e_1,a_je_j]=\lambda e_{1+j}$. Thus $a_j=\lambda$, $j\geq2$.
Moreover, for $j\geq2$, from $\phi_{f}(e_j,e_1)=-\phi_{f}(e_1,e_j)=-\lambda e_{1+j}$, we have
$[e_j,a_1e_1]=-\lambda e_{1+j}$. Thus $a_1=\lambda$. So $\phi=\lambda\sum_{i\geq1}e^{i}_{i}$.
\end{proof}
Here we characterize local and 2-local derivations of $\mathfrak{m}_0$ by the result of derivations.

\begin{thm}
$\mathrm{LDer}(\mathfrak{m}_0)=\mathrm{BLDer}(\mathfrak{m}_0)=\mathrm{Der}(\mathfrak{m}_0).$
\end{thm}
\begin{proof}
By Lemmas \ref{ld} and \ref{m0}, it is sufficient to prove
$\mathrm{LDer}_{k}(\mathfrak{m}_0)\subseteq\mathrm{Der}(\mathfrak{m}_0)$ for $k\geq0$.
By Lemma \ref{m0}, we prove that in the following  cases.
\begin{flushleft}
$\mathbf{Case\ 1.}\ k=0.$
\end{flushleft}
Suppose that $\Delta\in\mathrm{LDer}_{0}(\mathfrak{m}_0)$ is a local derivation. For $\Delta(e_i)$, $i\geq1$, by definition of local derivations,  there exist $a_i$, $b_i\in \F$, such that
\begin{eqnarray}\label{44}
    \Delta(e_1)&=&b_1e_{1}, \\
    \Delta(e_2)&=&a_2e_2, \\
   \Delta(e_i)&=&(a_i+(i-2)b_i)e_i,\ i\geq 3.
\end{eqnarray}
By linearity of $\Delta$, for $i\geq 3$,
\begin{equation}\label{11}
\Delta(e_1+e_2+e_i)=\Delta(e_1)+\Delta(e_2)+\Delta(e_i)=b_1e_{1}+a_2e_2+(a_i+(i-2)b_i)e_i.
\end{equation}
For $\Delta(e_1+e_2+e_i)$, by definition, there exist $a_{12i}$, $b_{12i}$, such that
\begin{equation}\label{22}
\Delta(e_1+e_2+e_i)=a_{12i}(e_2+e_i)+b_{12i}e_{1}+b_{12i}(i-2)e_i.
\end{equation}
Comparing Eqs. (\ref{11}) with (\ref{22}), we get $a_{2}+b_{1}(i-2)=a_i+(i-2)b_i$. From Eq. (2.1.3),
\begin{equation}\label{313}
\Delta(e_i)=(a_{2}+b_{1}(i-2))e_i,\ i\geq 3.
\end{equation}
From Eqs. (2.1.1), (2.1.2) and (\ref{313}) and Lemma \ref{m0}, we have
$$\Delta=a_2\sum_{i\geq2}e^{i}_i+b_1\bigg(e^{1}_{1}+\sum_{i\geq3}(i-2)e^{i}_i\bigg)\in\mathrm{Der}_0(\mathfrak{m}_0).$$

\begin{flushleft}
$\mathbf{Case\ 2.}\ k\geq1.$
\end{flushleft}
Suppose that $\Delta_k\in\mathrm{LDer}_{k}(\mathfrak{m}_0)$ is a local derivation. For $\Delta_k(e_i)$, $i\geq1$, by definition of local derivations,  there exist $a_{k;i}$, $b_{k;i}\in \F$, such that
\begin{equation*}
    \Delta_k(e_1)=a_{k;1}e_{1+k},\quad \mathrm{and}\quad
   \Delta_k(e_i)=b_{k;i}e_{i+k},\ i\geq 2.
\end{equation*}
For $\Delta_k(e_2+e_i)$, $i\geq 3$, there exist $b_{k;2,i}\in\F$, such that
\begin{equation}\label{111}
 \Delta_k(e_2+e_i)=b_{k;2,i}\sum_{j\geq2}e^{j+k}_j(e_2+e_i)=b_{k;2,i}(e_{2+k}+e_{i+k}).
\end{equation}
On the other hand,
\begin{equation}\label{1111}
\Delta_k(e_2+e_i)=\Delta_k(e_2)+\Delta_k(e_i)=b_{k;2}e_{2+k}+b_{k;i}e_{i+k}.
\end{equation}
Comparing Eqs. (\ref{111}) with (\ref{1111}), we have $b_{k;2}=b_{k;i}$ for $i\geq 3$.
So $\Delta_k=a_{k;1}e_{1}^{1+k}+b_{k;2}\sum_{i\geq2}e^{i+k}_i\in\mathrm{Der}_k(\mathfrak{m}_0)$.
\end{proof}

\begin{rem}
 The 2-local derivations of  $\mathfrak{m}_0$ were also studied in \cite{M0,l1}. In fact, there are a few non-linear 2-local derivations which are not derivations (see \cite{M0,l1}, for example).
For $m,q\in \N$, $\lambda\in \C$ and $\theta=(\theta_2,\ldots,\theta_m)\in \C^{m-1}$, define the mapping $\Omega^{(q,m)}_{\theta,\lambda}$ of $\mathfrak{m}_0$ by
$$\Omega^{(q,m)}_{\theta,\lambda}\left(\sum_{i=1}^{p}k_ie_i\right)=
\left\{
  \begin{array}{ll}
    \sum\limits_{i=2}^{p}\sum\limits_{j=2}^{m}k_i\theta_je_{i+j-2}, & \hbox{if $k_1\neq 0$;} \\
    \lambda k_qe_q, & \hbox{if $x=k_q e_q$ for some $q$ with $2<q\leq p$;} \\
    0, & \hbox{others.}
  \end{array}
\right.$$
In \cite{M0}, the author prove that every 2-local derivation (not necessarily linear) $\Delta$ of  $\mathfrak{m}_0$ is of the form $\Delta=D+\Omega^{(q,m)}_{\theta,\lambda}$ for some $D\in\mathrm{Der}(\mathfrak{m}_0)$ \cite[Theorem 4.2]{M0}.
 Using this theorem, we can  give another proof that if a 2-local derivation of $\mathfrak{m}_0$
is linear, then it is a derivation.
Suppose that $\Delta=D+\Omega^{(q,m)}_{\theta,\lambda}$ is a 2-local derivation (not necessarily linear), where $D\in\mathrm{Der}(\mathfrak{m}_0)$ and $\Omega^{(q,m)}_{\theta,\lambda}$.
Then
$$\Omega^{(q,m)}_{\theta,\lambda}(e_1)=\Omega^{(q,m)}_{\theta,\lambda}(e_2)=0, \ \Omega^{(q,m)}_{\theta,\lambda}(e_q)=\lambda e_q.$$
If $\Delta$ is linear, then $\Omega^{(q,m)}_{\theta,\lambda}$ is additive. Thus, we have
\begin{eqnarray*}
  &&\Omega^{(q,m)}_{\theta,\lambda}(e_1+e_2)=\sum_{j=2}^{m}\theta_j e_j=0, \\
  && \Omega^{(q,m)}_{\theta,\lambda}(e_2+e_q)=0=\lambda e_q.
\end{eqnarray*}
Then $\theta_2=\cdots=\theta_m=\lambda=0$. That is, $\Omega^{(q,m)}_{\theta,\lambda}=0$ and $\Delta=D\in\mathrm{Der}(\mathfrak{m}_0)$.
\end{rem}

\section{$\mathbb{N}$-graded Lie algebra  $L_{1}$}


Different from $\mathfrak{m}_0$ and $\mathfrak{m}_2$, the derivarion algebra of $L_1$ has been described as the
inner derivation algebra of $\F e_0\ltimes L_1$.

\begin{lem}\cite{l1}\label{L'}
$\mathrm{Der}(L_1)=\mathrm{ad}_{L_1}(\F e_0\ltimes L_1)$, where
$\F e_0\ltimes L_1$ has the bracktes: $[e_i,e_j]=(j-i)e_{i+j}$ for all $i,j\geq0$.
\end{lem}

\begin{thm}\label{l}
$\mathrm{IBDer}(L_1)=\mathrm{BDer}(L_1)$.
\end{thm}
\begin{proof}
Suppose that $f$ is a biderivation of $L_1$. By Lemmas \ref{C} and \ref{L'}, there exists a mapping $\varphi_f:\ L_1\rightarrow L_1\oplus\F e_0$ such that, for $i,j\geq 1$,
\begin{equation}\label{L}
f(e_i,e_j)=[\varphi_f(e_i),e_j]=-[\varphi_f(e_j),e_i].
\end{equation}
Suppose that $\|f\|=\|\varphi_f\|=k$.
If $k\leq-1$, we set $\varphi_f(e_{-k+i})=\lambda_ie_{i}$, $i\geq 0$.
If $k\geq0$, we set $\varphi_f(e_i)=\lambda_ie_{k+i}$, $i\geq 1$.
Then, if $k\neq 0$, we have $\lambda_i=0$ by taking $i=j$ in Eq. (\ref{L}).
So $f=\varphi_f=0$.
Now we suppose that $k=0$. Then, from Eq. (\ref{L}),  we get $\lambda_i=\lambda_j$ for $i,j\geq 1$. Set $\lambda_i=\lambda$ for $i\geq 1$.
Then $f=f_\lambda$ is an inner biderivation.
\end{proof}

\begin{thm}
A linear mapping of $L_1$ is commuting if and only if it is a scalar multiplication mapping of $L_1$.
\end{thm}
\begin{proof}
The `if' direction is easy to verify. We now prove the `only if' direction.
Suppose that $\phi$ is linear commuting mapping of $L_1$. By Lemma \ref{linear}, $\phi$ defines a biderivation $\phi_{f}\in\mathrm{BDer}(L_1)$.
By Theorem \ref{l}, for all $x,y\in L_1$, $\phi_{f}(x,y)=[x,\phi(y)]=\lambda[x,y]$ for some $\lambda\in\mathbb{F}$.
Then $[x,\phi(y)-\lambda y]=0$. That means $\phi(y)-\lambda y\in \mathrm{C}(L_1)=0$. The proof is complete.
\end{proof}
From the result of the first cohomology of $L_1$  with coefficients in the adjoint module \cite{L1}, we get the derivations of $L_1$  by considering the inner derivations in the following lemma.

\begin{lem}\label{LL1}
$\mathrm{dim}\ \mathrm{Der}_{k}(L_1)=\left\{
                    \begin{array}{ll}
                      0, & \hbox{$k\leq-1$;} \\
                      1, & \hbox{$k\geq0$.}
                    \end{array}
                  \right.
$
In particular, in $\mathrm{Der}_{k}(L_1)$ for $k\geq0$, a basis is $\sum_{i\geq1,i\neq k}(i-k)e^{k+i}_i$.
\end{lem}

Here we characterize local and 2-local derivations of $L_1$ by the result of derivations.

\begin{thm}
$\mathrm{LDer}(L_1)=\mathrm{BLDer}(L_1)=\mathrm{Der}(L_1).$
\end{thm}
\begin{proof}
By Lemmas \ref{ld} and \ref{LL1}, it is sufficient to prove
$\mathrm{LDer}_{k}(L_1)\subseteq\mathrm{Der}(L_1)$ for $k\geq0$.
Suppose that $\Delta_k\in\mathrm{Der}_k(L_1)$ is a local derivation. Fix a $i_0: i_{0}\geq 1$ and $i_{0}\neq k$. For $\Delta_k(e_{i_{0}})$, by the definition of local derivations, there exists $a_{k;i_0}\in\F$, such that
$$\Delta_k(e_{i_{0}})=a_{k;i_0}\sum_{i\geq1,i\neq k}((i-k)e^{k+i}_i)(e_{i_{0}})=a_{k;i_0}(i_0-k)e_{k+i_0}.$$
Similarly, for any $j: j\geq 1$ and  $j\neq k, i_0$, there exist $a_{k;j}\in\F$ such that
$$\Delta_k(e_j)=a_{k;j}\sum_{i\geq1,i\neq k}((i-k)e^{k+i}_i)(e_j)=a_{k;j}(j-k)e_{k+j}.$$
Moreover,
\begin{equation}\label{x}
\Delta_k(e_{i_{0}}+e_j)=\Delta_k(e_{i_{0}})+\Delta_k(e_j)=a_{k;i_0}(i_0-k)e_{k+i_0}+a_{k;j}(j-k)e_{k+j}.
\end{equation}
For $\Delta_k(e_{i_{0}}+e_j)$, by definition, there exist $a_{k;i_0,j}\in\F$ such that
\begin{equation}\label{y}
 \Delta_k(e_{i_{0}}+e_j)=a_{k;i_0,j}((i_0-k)e_{k+i_0}+(j-k)e_{k+j}).
\end{equation}
Comparing Eqs (\ref{x}) with (\ref{y}), we have $a_{k;j}=a_{k;i_0}$ for any $j\geq1$ and $j\neq k$.
So $\Delta_k=a_{k;i_0}\sum_{i\geq1,i\neq k}(i-k)e^{k+i}_i\in\mathrm{Der}_k(L_1).$
\end{proof}
\begin{rem}
 The 2-local derivations of  $L_1$ were also studied in \cite{l1}. In particular,
 authors showed that every 2-local derivation (not necessarily linear) is derivation.
\end{rem}

\section{$\mathbb{N}$-graded Lie algebra  $\mathfrak{m}_2$}


 From the result of the first cohomology of $\mathfrak{m}_2$  with coefficients in the adjoint module \cite[Theorem 2]{m2}, we get the derivations of $\mathfrak{m}_2$  by considering the inner derivations in the following lemma.
\begin{lem}\label{M21}
$\mathrm{dim}\ \mathrm{Der}_{k}(\mathfrak{m}_2)=\left\{
                    \begin{array}{ll}
                      0, & \hbox{$k\leq-1$;} \\
                      1, & \hbox{$k=0,1$;} \\
                      2, & \hbox{$k\geq2$.}
                    \end{array}
                  \right.
$
In particular,

(1) in $\mathrm{Der}_{0}(\mathfrak{m}_2)$, a basis is $\sum_{i\geq1}ie^{i}_i$;

(2) in $\mathrm{Der}_{1}(\mathfrak{m}_2)$, a basis is $\sum_{i\geq2}e^{i+1}_i$;

(3) in $\mathrm{Der}_{2}(\mathfrak{m}_2)$, a basis is $\sum_{i\geq2}e^{i+2}_i$, $e^{3}_{1}-\sum_{i\geq3}e^{i+2}_{i}$;

(4) in $\mathrm{Der}_{k}(\mathfrak{m}_2)$ for $k\geq3$, a basis is $e^{k+1}_1+e^{k+2}_2$,
$-\frac{1}{2}e^{k+1}_1+\frac{1}{2}e^{k+2}_2+\sum_{i\geq3}e^{i+k}_i$.
\end{lem}

In order to describe the derivation algebra of $\mathfrak{m}_2$, we denote by $$\mathfrak{\widetilde{m}}_2=\mathfrak{m}_2\oplus\mathrm{span}\{x_{0},x_i,i\geq2\}$$
 the Lie algebra with brackets:
$$[e_1,e_i]=e_{i+1},\ [e_2,e_j]=e_{j+2},\ [x_0,x_i]=ix_i,\ [x_2,x_j]=\frac{1}{2}e_{j+2},\
[x_0,e_{i-1}]=(i-1)e_{i-1},$$
$$[x_2,e_i]=e_{i+2},\ [x_j,e_1]=-\frac{1}{2}e_{j+1},\ [x_j,e_2]=\frac{1}{2}e_{j+2},\ [x_i,e_j]=e_{i+j},$$
where  $i\geq2$, $j\geq3$.
Obviously, $\mathfrak{\widetilde{m}}_2=\bigoplus^{\infty}_{i=0}(\mathfrak{\widetilde{m}}_2)_i$ is a $\mathbb{Z}$-graded Lie algebra with the graded component
$(\mathfrak{\widetilde{m}}_2)_i=\left\{
                                  \begin{array}{ll}
                                    \F x_{0}, & \hbox{$i=0$;} \\
                                     \F e_{1}, & \hbox{$i=1$;} \\
                                    \mathrm{span}\{e_{i},x_{i}\}, & \hbox{$i\geq 2$.}
                                  \end{array}
                                \right.$
By Lemma \ref{m2}, $\mathrm{Der}(\mathfrak{m}_2)\cong\mathfrak{\widetilde{m}}_2$.

\begin{thm}\label{m2}
$\mathrm{dim}\ \mathrm{BDer}_{k}(\mathfrak{m}_2)=\left\{
                    \begin{array}{ll}
                      1, & \hbox{$k\geq0$;} \\
                     0, & \hbox{$k\leq-1$.}
                    \end{array}
                  \right.
$
In particular, for $k\geq0$, $\mathrm{BDer}_{k}(\mathfrak{m}_2)$ is spanned by $\sum_{i\geq2}e^{1,i}_{k+i+1}+\sum_{i\geq3}e^{2,i}_{k+i+2}$.
\end{thm}

\begin{proof}
By Lemma \ref{C},
for $f\in \mathrm{BDer}_{k}(\mathfrak{m}_2)$, $k\in\Z$, there exists a  linear mapping $\varphi_f:\ \mathfrak{m}_2\rightarrow \mathfrak{\widetilde{m}}_2$ such that, for $i,j\geq1$,
\begin{equation}\label{mb2}
f(e_i,e_j)=[\varphi_f(e_i),e_j]=-[\varphi_f(e_j),e_i],
\end{equation}
where $\|f\|=\|\varphi_f\|=k$. Now we determine $\varphi_f$ in different weight $k$.

\begin{flushleft}
$\mathbf{Case\ 1.}\ k\leq-2.$
\end{flushleft}
In this case, $\varphi_f(e_1)=\cdots=\varphi_f(e_{-k-1})=0$. Set $\varphi_f(e_{-k})=\alpha x_{0}$, $\varphi_f(e_{-k+1})=\beta e_1$ and $\varphi_f(e_{-k+i})=\lambda_ie_i+\mu_ix_i$ for $i\geq 2$.
Taking $j=1$ in Eq. (\ref{mb2}), we get
$[\varphi_f(e_{i}),e_1]=0$.
Thus, we have
\begin{eqnarray*}
&&[\varphi_f(e_{-k}),e_1]=[\alpha x_{0},e_1]=\alpha e_1=0, \\
&&[\varphi_f(e_{-k+2}),e_1]=[\lambda_2e_2+\mu_2x_2,e_1]=-\lambda_2e_3=0,\\
&&[\varphi_f(e_{-k+i}),e_1]=[\lambda_ie_i+\mu_ix_i,e_1]=-(\lambda_i+\frac{1}{2}\mu_i)e_{i+1}=0,
\end{eqnarray*}
where $i\geq 3$. So $\alpha=\lambda_2=\lambda_i+\frac{1}{2}\mu_i=0$ for $i\geq 3$. Taking $i=j$ in Eq. (\ref{mb2}), we get
$[\varphi_f(e_{i}),e_i]=0$. Thus, we have
\begin{eqnarray*}
  &&[\varphi_f(e_{-k+1}),e_{-k+1}]=[\beta e_1,e_{-k+1}]=\beta e_{-k+2}=0, \\
  &&[\varphi_f(e_{-k+i}),e_{-k+i}]=[\lambda_ie_i+\mu_ix_i,e_{-k+i}]=\mu_ie_{-k+2i}=0,
\end{eqnarray*}
where $i\geq 2$. So $\beta=\mu_i=0$ for $i\geq 2$. Thus $f=\varphi_f=0$.

\begin{flushleft}
$\mathbf{Case\ 2.}\ k=-1.$
\end{flushleft}
In this case, we set $\varphi_f(e_1)=\alpha x_{0}$, $\varphi_f(e_2)=\beta e_1$ and $\varphi_f(e_i)=\lambda_{i-1}e_{i-1}+\mu_{i-1}x_{i-1}$ for $i\geq 3$.
Taking $i=j$ in Eq. (\ref{mb2}), we get
$[\varphi_f(e_{i}),e_i]=0$. Thus, we have
\begin{eqnarray*}
 &&[\varphi_f(e_{1}),e_{1}]=[\alpha x_{0},e_{1}]=\alpha e_{1}=0, \\
 &&[\varphi_f(e_{2}),e_{2}]=[\beta e_{1},e_{2}]=\beta e_3=0,\\
&&[\varphi_f(e_{3}),e_{3}]=[\lambda_{2}e_{2}+\mu_{2}x_{2},e_{3}]=(\lambda_2+\mu_{2})e_{5}=0,\\
&&[\varphi_f(e_{i}),e_{i}]=[\lambda_{i-1}e_{i-1}+\mu_{i-1}x_{i-1},e_{i}]=\mu_{i-1}e_{2i-1}=0,
\end{eqnarray*}
where $i\geq4$. So $\alpha=\beta=\lambda_2+\mu_{2}=\mu_{i}=0$ for $i\geq 3$.
Taking $j=1$ in Eq. (\ref{mb2}), we get
$[\varphi_f(e_{i}),e_1]=-[\varphi_f(e_{1}),e_i]=0$.
Thus, we have
\begin{eqnarray*}
 &&[\varphi_f(e_3),e_1]=[\lambda_2e_2+\mu_2x_2,e_1]=-\lambda_2e_3=0, \\
  &&[\varphi_f(e_i),e_1]=[\lambda_{i-1}e_{i-1},e_1]=-\lambda_{i-1}e_i=0,
\end{eqnarray*}
where $i\geq 4$. So $\lambda_i=0$ for $i\geq 2$. Thus $f=\varphi_f=0$.

\begin{flushleft}
$\mathbf{Case\ 3.}\ k=0.$
\end{flushleft}
In this case, we set $\varphi_f(e_1)=\alpha e_{1}$ and $\varphi_f(e_i)=\lambda_{i}e_{i}+\mu_{i}x_{i}$ for $i\geq 2$.
Taking $i=j$ in Eq. (\ref{mb2}), we get
$[\varphi_f(e_{i}),e_i]=0$. Thus, we have
\begin{equation*}
[\varphi_f(e_{i}),e_{i}]=[\lambda_ie_i+\mu_ix_{i},e_{i}]=\mu_i e_{2i}=0,
\end{equation*}
where $i\geq2$. So $\mu_{i}=0$ for $i\geq 2$.
Taking $i=1$ in Eq. (\ref{mb2}), we get
$[\varphi_f(e_{1}),e_j]=-[\varphi_f(e_{j}),e_1]$.
Thus, we have
\begin{equation*}
[\varphi_f(e_{1}),e_j]=-[\varphi_f(e_{j}),e_1]=\alpha e_{j+1}=\lambda_je_{j+1},
\end{equation*}
where $j\geq 2$. So $\lambda_i=\alpha$ for $i\geq 2$. Thus,
\begin{eqnarray*}
&&f(e_1,e_i)=[\varphi_f(e_1),e_i]=[\alpha e_{1},e_i]=\alpha e_{i+1},\ i\geq2, \\
&&f(e_2,e_i)=[\varphi_f(e_2),e_i]=[\alpha e_2,e_i]=\alpha e_{i+2} ,\ i\geq 3,\\
&&f(e_i,e_j)=[\varphi_f(e_i),e_j]=[\alpha e_{i},e_j]=0  ,\ j>i\geq 3.
\end{eqnarray*}
That is,
 $f=\alpha\left(\sum_{i\geq2}e^{1,i}_{i+1}+\sum_{i\geq3}e^{2,i}_{i+2}\right)$.

\begin{flushleft}
$\mathbf{Case\ 4.}\ k=1.$
\end{flushleft}
In this case, we set $\varphi_f(e_i)=\lambda_{i}e_{i+1}+\mu_{i}x_{i+1}$ for $i\geq 1$.
Taking $i=j$ in Eq. (\ref{mb2}), we get
$[\varphi_f(e_{i}),e_i]=0$. Thus, we have
\begin{eqnarray*}
 &&[\varphi_f(e_{1}),e_{1}]=[\lambda_{1}e_{2}+\mu_{1}x_{2},e_{1}]=-\lambda_1 e_{3}=0, \\
 &&[\varphi_f(e_{2}),e_{2}]=[\lambda_{2}e_{3}+\mu_{2}x_{3},e_{2}]=(\frac{1}{2}\mu_2-\lambda_2) e_5=0,\\
&&[\varphi_f(e_{i}),e_{i}]=[\lambda_{i}e_{i+1}+\mu_{i}x_{i+1},e_{i}]=\mu_{i}e_{2i+1}=0,
\end{eqnarray*}
where $i\geq 3$. So $\lambda_2=\frac{1}{2}\mu_2$, $\lambda_1=\mu_i=0$ for $i\geq 3$. Taking $j=1$ in
 Eq. (\ref{mb2}), we get $[\varphi_f(e_i),e_1]=-[\varphi_f(e_1),e_i]$. Thus, we have
\begin{equation*}
[\varphi_f(e_i),e_1]=-[\varphi_f(e_1),e_i]=-(\lambda_i+\frac{1}{2}\mu_i)e_{i+2}=-\mu_1e_{i+2},
\end{equation*}
where $i\geq 2$. So $\lambda_i=\mu_1-\frac{1}{2}\mu_i$ for $i\geq2$.
Thus,
\begin{eqnarray*}
&&f(e_1,e_i)=[\varphi_f(e_1),e_i]=[\mu_1 x_{2},e_i]=\mu_1 e_{i+2},\ i\geq2, \\
&&f(e_2,e_i)=[\varphi_f(e_2),e_i]=[\frac{1}{2}\mu_1 e_3+\mu_1x_3,e_i]=\mu_1 e_{i+3} ,\ i\geq 3,\\
&&f(e_i,e_j)=[\varphi_f(e_i),e_j]=[\mu_1 e_{i+1},e_j]=0  ,\ j>i\geq 3.
\end{eqnarray*}
That is,
 $f=\mu_1\left(\sum_{i\geq2}e^{1,i}_{i+2}+\sum_{i\geq3}e^{2,i}_{i+3}\right)$.

\begin{flushleft}
$\mathbf{Case\ 5.}\ k\geq2.$
\end{flushleft}
In this case, we set $\varphi_f(e_i)=\lambda_{i}e_{i+k}+\mu_{i}x_{i+k}$ for $i\geq 1$.
Taking $i=j$ in Eq. (\ref{mb2}), we get
$[\varphi_f(e_{i}),e_i]=0$. Thus we have
\begin{eqnarray*}
 &&[\varphi_f(e_{1}),e_{1}]=[\lambda_{1}e_{k+1}+\mu_{1}x_{k+1},e_{1}]=-(\lambda_1+\frac{1}{2}\mu_1) e_{k+2}=0, \\
 &&[\varphi_f(e_{2}),e_{2}]=[\lambda_{2}e_{k+2}+\mu_{2}x_{k+2},e_{2}]=(\frac{1}{2}\mu_2-\lambda_2) e_{k+4}=0,\\
&&[\varphi_f(e_{i}),e_{i}]=[\lambda_{i}e_{k+i}+\mu_{i}x_{k+i},e_{i}]=\mu_{i}e_{2i+k}=0,
\end{eqnarray*}
where $i\geq 3$. So $\lambda_1=-\frac{1}{2}\mu_1$, $\lambda_2=\frac{1}{2}\mu_2$, $\mu_i=0$ for $i\geq 3$. Taking $j=1$ in
 Eq. (\ref{mb2}),
 we get $[\varphi_f(e_i),e_1]=-[\varphi_f(e_1),e_i]$. Thus, we have
\begin{eqnarray*}
&& [\varphi_f(e_2),e_1]=-[\varphi_f(e_1),e_2]=-\mu_2e_{k+3}=-\mu_1e_{k+3},\\
  && [\varphi_f(e_i),e_1]=-[\varphi_f(e_1),e_i]=-\lambda_ie_{k+i+1}=-\mu_1e_{k+i+1},
\end{eqnarray*}
where $i\geq 3$. So $\mu_1=\mu_2=\lambda_i$ for $i\geq 3$.
Thus,
\begin{eqnarray*}
&&f(e_1,e_i)=[\varphi_f(e_1),e_i]=[-\frac{1}{2}\mu_1e_{k+1}+\mu_1 x_{k+1},e_i]=\mu_1 e_{k+i+1},\ i\geq2, \\
&&f(e_2,e_i)=[\varphi_f(e_2),e_i]=[\frac{1}{2}\mu_1 e_{k+2}+\mu_1x_{k+2},e_i]=\mu_1 e_{k+i+2} ,\ i\geq 3,\\
&&f(e_i,e_j)=[\varphi_f(e_i),e_j]=[\mu_1 e_{k+i},e_j]=0  ,\ j>i\geq 3.
\end{eqnarray*}
That is,
 $f=\mu_1\left(\sum_{i\geq2}e^{1,i}_{k+i+1}+\sum_{i\geq3}e^{2,i}_{k+i+2}\right)$.

In conclusion,
$$f=\left\{
                    \begin{array}{ll}
                      \lambda_k\left(\sum_{i\geq2}e^{1,i}_{k+i+1}+\sum_{i\geq3}e^{2,i}_{k+i+2}\right), & \hbox{$\lambda_k\in\mathbb{F},\ k\geq0$;} \\
                      0, & \hbox{$k\leq-1$.}
                    \end{array}
                  \right.
$$
The proof is complete.
\end{proof}

\begin{cor}
$\mathrm{BDer}_{0}(\mathfrak{m}_2)=\mathrm{IBDer}(\mathfrak{m}_2)$.
\end{cor}
\begin{thm}
A linear mapping of $\mathfrak{m}_2$ is commuting if and only if it is a scalar of $\sum_{i\geq1}e^{i}_{i}$.
\end{thm}
\begin{proof}
The `if' direction is easy to verify. We now prove the `only if' direction.
Suppose that $\phi$ is linear commuting mapping of weight $k$. That is to say that $\phi(e_i)=a_ie_{i+k}$ for $i\geq1$. By Lemma \ref{linear}, $\phi$ defines a biderivation $\phi_{f}\in\mathrm{BDer}_{k}(\mathfrak{m}_2)$.
If $k\leq-1$, then $\phi_{f}=0$ by Theorem \ref{m2}. So $[x,\phi(y)]=0$ for all $x,y\in\mathfrak{m}_2$. From $\phi(y)\in\mathrm{C}(\mathfrak{m}_2)=0$, we have $\phi=0$. If $k\geq0$, then
$\phi_{f}=\lambda(\sum_{i\geq2}e^{1,i}_{k+i+1}+\sum_{i\geq3}e^{2,i}_{k+i+2})$ for some $\lambda\in\mathbb{F}$.
\begin{flushleft}
$\mathbf{Case\ 1.}\ k\geq2.$
\end{flushleft}
From $\phi_{f}(e_j,e_1)=\lambda(\sum_{i\geq2}e^{1,i}_{k+i+1}+\sum_{i\geq3}e^{2,i}_{k+i+2})(e_j,e_1)$ for $j\geq3$, we have
$[e_j,a_1e_{k+1}]=-\lambda e_{1+j+k}$.
So $\lambda=0$. We have $\phi_{f}(x,y)=[x,\phi(y)]=0$ for all $x,y\in\mathfrak{m}_2$. So $\phi(y)\in \mathrm{C}(\mathfrak{m}_2)=0$. Then $\phi=0$.
\begin{flushleft}
$\mathbf{Case\ 2.}\ k=1.$
\end{flushleft}
From $\phi_{f}(e_2,e_1)=\lambda(\sum_{i\geq2}e^{1,i}_{k+i+1}+\sum_{i\geq3}e^{2,i}_{k+i+2})(e_2,e_1)$, we have
$[e_2,a_1e_{2}]=-\lambda e_{4}=0$.
So $\lambda=0$. We have $\phi_{f}(x,y)=[x,\phi(y)]=0$ for all $x,y\in\mathfrak{m}_2$. So $\phi(y)\in \mathrm{C}(\mathfrak{m}_2)=0$. Then $\phi=0$.
\begin{flushleft}
$\mathbf{Case\ 3.}\ k=0.$
\end{flushleft}
From $\phi_{f}(e_1,e_j)=\lambda(\sum_{i\geq2}e^{1,i}_{k+i+1}+\sum_{i\geq3}e^{2,i}_{k+i+2})(e_1,e_j)$ for $j\geq2$,
we have $[e_1,a_je_j]=\lambda e_{1+j}$. Thus $a_j=\lambda$, $j\geq2$.
Moreover, for $j\geq2$, from $\phi_{f}(e_j,e_1)=-\phi_{f}(e_1,e_j)=-\lambda e_{1+j}$, we have
$[e_j,a_1e_1]=-\lambda e_{1+j}$. Thus $a_1=\lambda$. So $\phi=\lambda\sum_{i\geq1}e^{i}_{i}$.
\end{proof}

Here we characterize local and 2-local derivations of $\mathfrak{m}_2$ by the result of derivations.

\begin{thm}
$\mathrm{LDer}(\mathfrak{m}_2)=\mathrm{BLDer}(\mathfrak{m}_2)=\mathrm{Der}(\mathfrak{m}_2).$
\end{thm}

\begin{proof}
By Lemmas \ref{ld} and \ref{M21}, it is sufficient to prove
$\mathrm{LDer}_{k}(\mathfrak{m}_2)\subseteq\mathrm{Der}(\mathfrak{m}_2)$ for $k\geq0$.
By Lemma \ref{M21}, we prove that in the following  cases.
\begin{flushleft}
$\mathbf{Case\ 1.}\ k=0.$
\end{flushleft}
Suppose that $\Delta\in\mathrm{LDer}_{0}(\mathfrak{m}_2)$ is a local derivation.  For
$\Delta(e_i)$, $i\geq1$, by definition, there exists $a_i\in\F$ such that
$$\Delta(e_i)=a_i\sum_{j\geq1}je^{j}_j(e_i)=ia_ie_i.$$
Then, for $i\geq 2$,
\begin{equation}\label{23}
\Delta(e_1+e_i)=\Delta(e_1)+\Delta(e_i)=a_1e_1+ia_ie_i.
\end{equation}
For $\Delta(e_1+e_i)$, $i\geq2$, by definition, there exists $a_{1,i}\in\F$ such that
\begin{equation}\label{24}
 \Delta(e_1+e_i)=a_{1,i}\sum_{j\geq1}je^{j}_j(e_1+e_i)=a_{1,i}(e_1+ie_i).
\end{equation}
Comparing Eqs. (\ref{23}) with (\ref{24}), we have $a_i=a_1$ for $i\geq 1$. So $\Delta=a_1\sum_{i\geq1}ie^{i}_i\in
\mathrm{Der}_0(\mathfrak{m}_2)$.
\begin{flushleft}
$\mathbf{Case\ 2.}\ k=1.$
\end{flushleft}
Suppose that $\Delta\in\mathrm{LDer}_{1}(\mathfrak{m}_2)$ is a local derivation.  For
$\Delta(e_i)$, $i\geq2$, by definition, there exists $a_i\in\F$ such that
$$\Delta(e_i)=a_i\sum_{j\geq2}e^{j+1}_{j}(e_i)=a_ie_{i+1}.$$
Then, for $i\geq 3$,
\begin{equation}\label{33}
\Delta(e_2+e_i)=\Delta(e_2)+\Delta(e_i)=a_2e_3+a_ie_{i+1}.
\end{equation}
For $\Delta(e_2+e_i)$, $i\geq3$, by definition, there exists $a_{2,i}\in\F$ such that
\begin{equation}\label{34}
 \Delta(e_2+e_i)=a_{2,i}\sum_{j\geq2}e^{j+1}_j(e_2+e_i)=a_{2,i}(e_3+e_{i+1}).
\end{equation}
Comparing Eqs. (\ref{33}) with (\ref{34}), we have $a_i=a_2$ for $i\geq 2$. So $\Delta=a_2\sum_{i\geq2}e^{i+1}_i\in
\mathrm{Der}_1(\mathfrak{m}_2)$.

\begin{flushleft}
$\mathbf{Case\ 3.}\ k=2.$
\end{flushleft}
Suppose that $\Delta\in\mathrm{LDer}_{2}(\mathfrak{m}_2)$ is a local derivation.  For
$\Delta(e_i)$, $i\geq1$, by definition, there exist $a_i,b_i\in\F$ such that
\begin{eqnarray*}
 \Delta(e_1) &=& b_1e_3, \\
  \Delta(e_2) &=&a_2e_4, \\
 \Delta(e_i) &=&(a_i-b_i)e_{i+2},\ i\geq 3.
\end{eqnarray*}
Then, for $i\geq 3$,
\begin{equation}\label{43}
\Delta(e_1+e_2+e_i)=\Delta(e_1)+\Delta(e_2)+\Delta(e_i)=b_1e_3+a_2e_4+(a_i-b_i)e_{i+2}.
\end{equation}
For $\Delta(e_1+e_2+e_i)$, $i\geq3$, by definition, there exist $a_{1,2,i},b_{1,2,i}\in\F$ such that
\begin{equation}\label{44}
 \Delta(e_1+e_2+e_i)=b_{1,2,i}e_3+a_{1,2,i}e_4+(a_{1,2,i}-b_{1,2,i})e_{i+2}.
\end{equation}
Comparing Eqs. (\ref{43}) with (\ref{44}), we have $a_i-b_i=a_2-b_1$ for $i\geq 3$. So $\Delta=a_2\sum_{i\geq2}e^{i+2}_i+b_1(e^{3}_{1}-\sum_{i\geq3}e^{i+2}_{i})\in
\mathrm{Der}_2(\mathfrak{m}_2)$.
\begin{flushleft}
$\mathbf{Case\ 4.}\ k\geq 3.$
\end{flushleft}
Suppose that $\Delta_{k}\in\mathrm{LDer}_{k}(\mathfrak{m}_2)$ is a local derivation.
For
$\Delta(e_i)$, $i\geq1$, by definition, there exist $a_{k;i},b_{k;i}\in\F$ such that
\begin{eqnarray*}
 \Delta(e_1) &=& (a_{k;1}-\frac{1}{2}b_{k;1})e_{1+k}, \\
  \Delta(e_2) &=&(a_{k;2}+\frac{1}{2}b_{k;2})e_{2+k}, \\
 \Delta(e_i) &=&b_{k;i}e_{i+k},\ i\geq 3.
\end{eqnarray*}
Then, for $i\geq 4$,
\begin{equation}\label{53}
\Delta(e_3+e_i)=\Delta(e_3)+\Delta(e_i)=b_{k;3}e_{3+k}+b_{k;i}e_{i+k}.
\end{equation}
For $\Delta(e_3+e_i)$, $i\geq4$, by definition, there exists $b_{k;3,i}\in\F$ such that
\begin{equation}\label{54}
 \Delta(e_3+e_i)=b_{k;3,i}(e_{3+k}+e_{i+k}).
\end{equation}
Comparing Eqs. (\ref{53}) with (\ref{54}), we have $b_{k;i}=b_{k;3}$ for $i\geq 3$.
Similarly, for
$\Delta(e_1+e_2+e_3)$ there exist $a_{k;1,2,3},b_{k;1,2,3}\in\F$ such that
\begin{eqnarray*}
  \Delta(e_1+e_2+e_3) &=&(a_{k;1,2,3}-\frac{1}{2}b_{k;1,2,3})e_{1+k}+(a_{k;1,2,3}+\frac{1}{2}b_{k;1,2,3})e_{2+k}+b_{k;1,2,3}e_{3+k}\\
  &=&    \Delta(e_1)+\Delta(e_2)+\Delta(e_3)\\
  &=& (a_{k;1}-\frac{1}{2}b_{k;1})e_{1+k}+(a_{k;2}+\frac{1}{2}b_{k;2})e_{2+k}+
   b_{k;3}e_{3+k}.
\end{eqnarray*}
So $(a_{k;2}+\frac{1}{2}b_{k;2})-(a_{k;1}-\frac{1}{2}b_{k;1})=b_{k;3}$.
Moreover, $\Delta_k=(a_{k;1}-\frac{1}{2}b_{k;1}+\frac{1}{2}b_{k;3})(e^{k+1}_1+e^{k+2}_2)+b_{k;3}(-\frac{1}{2}e^{k+1}_1+\frac{1}{2}e^{k+2}_2+
\sum_{i\geq3}e^{i+k}_i)\in\mathrm{Der}_{k}(\mathfrak{m}_2)$.
\end{proof}

\small\noindent \textbf{Acknowledgment}\\

The authors are supported by  NSF of Jilin Province (No. YDZJ202201ZYTS589), NNSF of China (Nos. 12271085, 12071405, 12001141) and the Fundamental Research Funds for the Central Universities.


\begin{thebibliography}{99}

\bibitem{AKR} Sh. A. Ayupov, K. K. Kudaybergenov, I. S. Rakhimov.
2-Local derivations on finite-dimensional Lie algebras.
\emph{Linear Algebra Appl.}     \textbf{474} (2015) 1--11.


\bibitem{l1} Sh. A. Ayupov, B. Yusupov.
2-Local derivations of infinite-dimensional Lie algebras.
\emph{J. Algebra Appl.} \textbf{19} (2020) 2050100, 12 pp.

\bibitem{b}M. Bre\u{s}ar.
 On generalized biderivations and related maps.
\emph{J. Algebra} \textbf{172} (1995) 764--786.


\bibitem{ex4} M. Bre\u{s}ar, K. M. Zhao.
Biderivations and commuting  linear maps on Lie algebras.
\emph{J. Lie Theory}  \textbf{28} (2018) 885--900.

\bibitem{ccc2021} Y. Chang, L. Y. Chen, Y. Cao. Super-biderivation of the generalized Witt Lie superalgebra $W(m,n;\underline{t})$.
\emph{Linear Multilinear Algebra} \textbf{69} (2021) 233--244.

\bibitem{ccz2019} Y. Chang, L. Y. Chen, X. Zhou. Biderivation and linear commuting maps on the restricted Cartan-type Lie algebras $H(n,\underline{1}).$
\emph{Comm. Algebra} \textbf{47} (2019) 1311--1326.

\bibitem{c2016} Z. X. Chen. Biderivations and linear commuting maps on simple generalized Witt algebras over a field.
\emph{Electron. J. Linear Algebra} \textbf{31} (2016) 1--12.

\bibitem{ex2} X. Cheng, M. J. Wang, J. C. Sun, H. L. Zhang.
Biderivations and linear commuting maps on the Lie algebra $\mathfrak{gca}$.
\emph{Linear Multilinear Algebra} \textbf{65} (2017) 2483--2493.

\bibitem{dt2019} Z. B. Ding, X. M. Tang. Biderivations of the Galilean conformal algebra and their applications.
\emph{Quaest. Math.} \textbf{42} (2019) 831--839.
\bibitem{e2021} D. Eremita. Biderivations and commuting linear maps on current Lie algebras.
\emph{J. Lie Theory} \textbf{31} (2021) 119--126.








\bibitem{fc}
 A. Fialowski.
 Classification of graded Lie algebras with two generators.
\emph{ Moscow Univ. Math. Bull.} \textbf{38} (1983) 76--79.


\bibitem{L1} A. Fialowski.
An example of formal deformations of Lie algebras.
\emph{NATO Conference on Deformation Theory of Algebras and Applications, Il Ciocco, Italy, 1986}, Proceedings. Kluwer, Dordrecht, \textbf{} (1988) 375--401.

\bibitem{ft} A. Fialowski, D.  Millionschikov.
Cohomology of graded Lie algebras of maximal class.
\emph{J. Algebra} \textbf{296} (2006) 157--176.

\bibitem{m0} A. Fialowski, F. Wagemann.
Cohomology and deformations of the infinite-dimensional filiform Lie algebra $\mathfrak{m}_0$.
\emph{J. Algebra} \textbf{318} (2007) 1002--1026.






\bibitem{m2} A. Fialowski, F. Wagemann.
Cohomology and deformations of the infinite-dimensional filiform Lie algebra $\mathfrak{m}_2$.
\emph{J. Algebra} \textbf{319} (2008) 5125--5143.


\bibitem{Fuks} D. B. Fuchs. Cohomology of infinite-dimensional Lie algebras. \emph{Contemporary Soviet Mathematics. Consultants Bureau, New York.} 1986.

\bibitem{jt2020} J. W. Jiang, X. M. Tang. Biderivations of the deformative Schr\"{o}dinger-Virasoro algebras.
\emph{Comm. Algebra} \textbf{48} (2020) 609--624.


\bibitem{K} R. V. Kadison. Local derivations. \emph{J. Algebra} \textbf{130} (1990) 494--509.









\bibitem{M1} D. V. Millionshchikov.
Cohomology of graded Lie algebras of maximal class with coefficients in the adjoint representation.
\emph{Proc. Steklov Inst. Math.} \textbf{263} (2008) 99--111.

\bibitem{M2} D. V. Millionshchikov.
Deformations of filiform Lie algebras and symplectic structures.
\emph{Proc. Steklov Inst. Math.} \textbf{252} (2006) 182--204.

\bibitem{P} P. \v{S}emrl.
 Local automorphisms and derivations on $\mathcal{B}(H)$.
\emph{Proc. Amer. Math. Soc.} \textbf{125} (1997), 2677--2680.



\bibitem{cc} A. Shalev, E. I. Zelmanov.
Narrow Lie Algebras: A coclass theory and a characterization of the Witt algebra.
\emph{J. Algebra} \textbf{189} (1997) 294--331.

\bibitem{M0} X. M. Tang.
2-Local derivations on the W-algebra $W(2,2)$.
\emph{J. Algebra Appl.} \textbf{20} (2021) 2150237, 13 pp.




\bibitem{tmc2020} L. M. Tang, L. Y. Meng, L. Y. Chen. Super-biderivations and linear super-commuting maps on the Lie superalgebras.
    \emph{Comm. Algebra} \textbf{48} (2020) 5076--5085.

\bibitem{ex1} X. M. Tang.
Biderivations of finite-dimensional complex simple Lie algebras.
\emph{Linear Multilinear Algebra} \textbf{66} (2018) 250--259.








\bibitem{ycc2021} J. X. Yuan, L. Y. Chen, Y. Cao. Super-biderivations of Cartan type Lie superalgebras.
\emph{Comm. Algebra} \textbf{49} (2021) 4416--4426.







\bibitem{zcz2020} X. D. Zhao, Y. Chang, X. Zhou, L. Y. Chen. Super-biderivations of the contact Lie superalgebra $K(m,n;\underline{t}).$
\emph{Comm. Algebra} \textbf{48} (2020) 3237--3248.



\end{thebibliography}
\end{document}